\documentclass{amsart}
\usepackage{bbm}
\usepackage{amsfonts}
\usepackage{mathrsfs}

\def\vep{\varepsilon}
\newtheorem{lem}{Lemma}[section]

\newtheorem{dfn}[lem]{Definition}
\newtheorem{pro}[lem]{Proposition}
\newtheorem{thm}[lem]{Theorem}
\newtheorem{cor}[lem]{Corollary}
\theoremstyle{remark}

\DeclareMathOperator\card{Card}
\numberwithin{equation}{section}

\begin{document}
\title[Hausdorff dimensions of sets related to Erd\"{o}s-R\'{e}nyi   averages]{Hausdorff dimensions of sets related to Erd\"{o}s-R\'{e}nyi averages in beta expansions}
\author{Haibo Chen}
 \address{School of Statistics and Mathematics, Zhongnan University of Economics and Law, Wu\-han, Hubei, 430073, China}
  \email{hiboo\_chen@sohu.com}

\begin{abstract}
Let $\beta>1$, $I$ be the unite interval $[0,1)$ and $\phi$ be an integer function defined on $\mathbb{N}\setminus\{0\}$ satisfying $1\leq\phi(n)\leq n$. Denote by $A_\phi(x,\beta)$ the Erd\"{o}s-R\'{e}nyi average of $x\in I$ associated with the function $\phi$ in $\beta$-expansion and $I_\beta$ the range of $A_\phi(x,\beta)$ for $x\in I$. For the level set 
\begin{align*}
	ER_\phi^\beta(\alpha)=\left\{x\in I\colon A_\phi(x,\beta)=\alpha\right\},\quad\text{where}\ \alpha\in I_\beta,
\end{align*}
in this paper we will determine its Hausdorff dimension under the assumption $\phi(n)\to\infty$ as $n\to\infty$ and $\phi$ is the integer part of some slowly varying sequence. Besides, a generalization to the classic work \cite{Be} of Besicovitch is also given in $\beta$-expansion.
\end{abstract}

\maketitle

\section{Introduction}
Let $\beta>1$ be a real number and $I$ be the unit interval $[0,1)$. Define the \emph{$\beta$-transformation} $T_\beta\colon I\to I$ as
\[T_\beta(x)=\beta x-[\beta x],\quad x\in I.\] 
Here, $[\cdot]$ is the floor function. It is well-known (see \cite{R}) that each $x\in I$ can be uniquely expanded into a finite or an infinite series as
\begin{align}\label{formula beta 1}
	x=\frac{\vep_1(x,\beta)}{\beta}+\frac{\vep_2(x,\beta)}{\beta^2}+\cdots+\frac{\vep_n(x,\beta)}{\beta^n}+\cdots,
\end{align}
where $\vep_n(x,\beta)=\big[\beta T_\beta^{n-1}(x)\big]$, $n\geq1$, is called the $n$-th digit of $x$ with respect to base $\beta$. For simplicity, we can also identify $x$ with the \emph{digit sequence}
\[\vep(x,\beta):=\big(\vep_1(x,\beta),\vep_2(x,\beta),\ldots\big).\]
 That is, we can rewrite \eqref{formula beta 1} as
\begin{align}\label{formula beta 2}
	x=\big(\vep_1(x,\beta),\vep_2(x,\beta),\ldots\big).
\end{align}
The formulas \eqref{formula beta 1} or \eqref{formula beta 2} is called the $\beta$-expansion of $x$ and the system $(I,T_\beta)$ is called the $\beta$-dynamical system.

It is clear that each $n$-th digit $\vep_n(x,\beta)$, $n\geq1$, of $x$ belongs to the alphabet $\Sigma=\{0,1,\ldots,\lceil\beta-1\rceil\}$, where $\lceil\cdot\rceil$ is the ceiling function. Denote by $\Sigma^\infty$ the set of infinite sequences with all digits from $\Sigma$. 
It is worth noting that not all sequences belong to $\Sigma^\infty$ would be the $\beta$-expansion of some $x\in I$.
Moreover, endowed with the metric 
\begin{align}
	d(\vep,\eta)=\beta^{-\min\{i\geq1\colon\vep_i\neq\eta_i\}},
\end{align}
where $\vep=(\vep_1,\vep_2,\ldots)\in\Sigma^\infty$ and $\eta=(\eta_1,\eta_2,\ldots)\in\Sigma^\infty$, the space $\Sigma^\infty$ is compact.

In the present paper, we will say the terminology word to mean a finite sequence, $n$-word to mean a finite sequence of length $n$ and meanwhile sequence to mean a infinite sequence for the sake of distinction.

Let $n\geq1$. We call an $n$-word $(\vep_1,\vep_2,\ldots,\vep_n)$ or a sequence $(\vep_1,\vep_2,\ldots,\vep_n,\ldots)$ \emph{$\beta$-admissible} if there exists an $x\in I$ such that the $\beta$-expansion of $x$ begins with this word or is just this sequence. Denote by $\Sigma_\beta^n$, $n\geq1$, the set of all $\beta$-admissible words with length $n$ and $\Sigma_\beta$ the set of all admissible sequences. That is,
\[\Sigma_\beta=\left\{\vep\in\Sigma^\infty\colon \vep\ \text{is the $\beta$-expansion of some $x\in I$}\right\}.\]

After the introduction of the concept of $\beta$-expansion, we would like to introduce the concept of Erd\"{o}s-R\'{e}nyi average in $\beta$-expansion. Let $x\in I$ and be of infinite $\beta$-expansion $\big(\vep_1(x,\beta),\vep_2(x,\beta),\ldots\big)$. Note that in the sequel we only need to deal with those numbers in $I$ with infinite $\beta$-expansions while the set of numbers with finite $\beta$-expansions is countable. Denoted by \[S_n(x,\beta)=\sum_{i=1}^n\vep_i(x,\beta),\quad n\geq1,\]
the sum of the first $n$ digits of $x$ or the $n$-th \emph{partial sum} of $x$. Let $\phi$ be an integer function defined on $\mathbb{N}\setminus\{0\}$ satisfying $1\leq\phi(n)\leq n$. Put
\begin{align}\label{ER result}
	I_{n,\phi(n)}(x,\beta)=\max_{0\leq i\leq n-\phi(n)}\big(S_{i+\phi(n)}(x,\beta)-S_i(x,\beta)\big)
\end{align}
and call it the \emph{$(n,\phi(n))$-Erd\"{o}s-R\'{e}nyi maximum partial sum} of $x$. Here, $S_0(x,\beta)=0$ is set by convention. Accordingly, call
\begin{align}
	A_{n,\phi(n)}(x,\beta)=\frac{I_{n,\phi(n)}(x,\beta)}{\phi(n)}
\end{align}
the \emph{$(n,\phi(n))$-Erd\"{o}s-R\'{e}nyi average} of $x$ and \begin{align}
	A_\phi(x,\beta)=\lim_{n\to\infty}A_{n,\phi(n)}(x,\beta)
\end{align} 
the \emph{Erd\"{o}s-R\'{e}nyi average} of $x$ associated with $\phi$ if the limit exists.
In particular, take $\phi(n)=1$, then \[I_{n,1}(x,\beta)=\max\{\vep_i(x,\beta),1\leq i\leq n\};\]
take  $\phi$ to be the identity function $\phi_I$ on $\mathbb{N}\setminus\{0\}$, i.e., $\phi_I(n)=n$, $n\geq1$, then $I_{n,\phi_I}(x,\beta)=S_n(x,\beta)$. So, we have
\begin{align}\label{AnnSn}
	A_{n,\phi_I}(x,\beta)=\frac{S_n(x,\beta)}{n}\quad\text{and}\quad A_{\phi_I}(x,\beta)=\lim_{n\to\infty}\frac{S_n(x,\beta)}{n}.
\end{align}
Thus, the Erd\"{o}s-R\'{e}nyi average is a more general concept than the usual algebraic average. In addition, in what follows we will write, respectively, $A_n(x,\beta)$ instead of $A_{n,\phi_I}(x,\beta)$ and $A(x,\beta)$ instead of $A_{\phi_I}(x,\beta)$ for brevity.

Erd\"{o}s-R\'{e}nyi average was first introduced by P. Erd\"{o}s and A. R\'{e}nyi \cite{ER} in 1970, which gave a pioneering work, by establishing a kind of new strong law of large numbers, on the limit behaviors of the length of the longest run of heads in $n$~independent Bernoulli trials. After that, many work emerged during the last several decades by considering, for examples, the asymptotic distribution of the length of the longest head run~\cite{Deh,ERev,GSW}, the appearances of long repetitive sequences in random sequences~\cite{GO}, the rate of convergence for a stationary sequence~\cite{No} and the case of renewal counting process~\cite{St}, etc.  In this paper, we would like to study the level sets described by the Erd\"{o}s-R\'{e}nyi average and determine their Hausdorff dimensions, which generalizes the work in~\cite{CDL} form binary expansion to $\beta$-expansion. In addition, some generalizations to the classic work \cite{Be} of Besicovitch in $\beta$-expansion are included as well. 
 
Let $\beta>1$. It is well-known (see \cite{R}) that there exists a unique invariant measure $\mu_\beta$, which is equivalent to the Lebesgue measure $\mathcal{L}$, when $\beta$ is not an integer. Moreover, $T_\beta$ is ergodic with respect to $\mu_\beta$ (see \cite{DK}). So, by Birkhoff's ergodic theorem, we have
\begin{align}\label{Ax beta}
	A(x,\beta)=\lim_{n\to\infty}\frac{1}{n}\sum_{i=0}^{n-1}\vep_1(T_\beta^ix,\beta)=\int_I\vep_1(x,\beta)d\mu_\beta=:\alpha^\ast(\beta),\quad \mu_\beta\text{-a.e.}\ x\in I.
\end{align}
Note that if we replace $\mu_\beta$ by the Lebesgue measure, then \eqref{Ax beta} is also valid when $\beta$ is an integer. In this case, we have $\alpha^\ast(\beta)=(\beta-1)/2$.  

Denote
\[\Lambda(\beta)=\sup_{x\in I}\bar{A}(x,\beta)\quad\text{and}\quad I_\beta=[0,\Lambda(\beta)],\]
where $\bar{A}(x,\beta)=\limsup_{n\to\infty}A_n(x,\beta)$, see also the definition~\eqref{definition axb} in Section 3. It is clear that $I_\beta=[0,\beta-1]$ when $\beta$ is an integer and $I_\beta\subset[0,\lceil\beta-1\rceil]$ when $\beta$ is not an integer. For a given $\alpha\in I_\beta$, define the level set
\begin{align}
  ER_\phi^\beta(\alpha)=\{x\in I\colon A_\phi(x,\beta)=\alpha\}.
\end{align}
In the following, we will show the Hausdorff dimension of  $ER_\phi^\beta(\alpha)$ after introducing some notations.

Let $\beta>1$, $\alpha\in I_\beta$, $n\geq1$ and $\delta>0$. Denote
\[H^\beta(\alpha,n,\delta)=\left\{(\vep_1,\vep_2,\ldots,\vep_n)\in \Sigma_\beta^n\colon n(\alpha-\delta)<\sum_{i=1}^n\vep_i<n(\alpha+\delta)\right\}\]
and $h^\beta(\alpha,n,\delta)=\card H^\beta(\alpha,n,\delta)$, where the symbol $\card$ denotes the cardinality of a set. 

Let $(\vep_1^*(1,\beta),\vep_2^*(1,\beta),\ldots)$ be the infinite $\beta$-expansion of 1 introduced in the beginning of Section 2. For each $m$ with $\vep_m^*(1,\beta)\geq1$, define $\beta_m=\beta_m(\beta)$ to be the unique positive root of the equation
\begin{align}\label{definition beta m}
	1=\frac{\vep_1^*(1,\beta)}{\beta_m^1}+\frac{\vep_2^*(1,\beta)}{\beta_m^2}+\cdots+\frac{\vep_m^*(1,\beta)}{\beta_m^m}.
\end{align}
Then we have $\beta_m<\beta$ and $\beta_m$ increases to $\beta$ as $m\to\infty$. Here and in the sequel, we always assume that $m$ takes value in $\{m\colon\vep_m^*(1,\beta)\geq1\}$ unless otherwise noted. Moreover, if $m_1<m_2$ and both of $\beta_{m_1}$ and $\beta_{m_2}$ are roots of the equation \eqref{definition beta m}, then the increasing property
\[\Sigma_{\beta_{m_1}}\subset\Sigma_{\beta_{m_2}}\subset\Sigma_\beta\]
holds by (2) in Theorem \ref{theorem parry}.

Then, define the $\beta$-adic entropy function as
\begin{align}\label{definiton h alpha}
	h^\beta(\alpha)=\lim_{\delta\to0}\lim_{m\to\infty}\liminf_{n\to\infty}\frac{\log h^{\beta_m}(\alpha,n,\delta)}{(\log \beta)n}=\lim_{\delta\to0}\lim_{m\to\infty}\limsup_{n\to\infty}\frac{\log h^{\beta_m}(\alpha,n,\delta)}{(\log \beta)n}.
\end{align}
Note that in \eqref{definiton h alpha} the limits for $\delta$ and $m$ both exist since $h^{\beta_m}(\alpha,n,\delta)$ is increasing for these two variables and the definition is valid since the second equality holds according to Proposition 4.2 in \cite{TWWX}. 

The function $\phi$, we focus on in this paper, is limited to a kind of special sequence, called slowly varying sequence. For the definition and corresponding properties, one can see Definition~\ref{definition slow increase} and Lemma~\ref{lem increse} in Section 6. Now, we would like to state the following main result in the present paper where $\dim_H$ denotes the Hausdorff dimension of a set.

\begin{thm}\label{theorem main theorem}
	Let $\beta>1$ and $\alpha\in I_\beta$. Assume the sequence $\{\theta(n)\}_{n\geq1}$ is slowly varying and $\theta(n)\to\infty$ as $n\to\infty$. If $\phi(n)=[\theta(n)]$, $n\geq1$, then we have
	\begin{align}\label{formula main}
		\dim_HER_\phi^\beta(\alpha)=\begin{cases} h^\beta(\alpha),\indent &0\leq\alpha\leq\alpha^\ast(\beta);\\
			1,\indent &\alpha^\ast(\beta)<\alpha\leq\Lambda(\beta).\end{cases} 
	\end{align}
\end{thm}

In particular, for the case that $\beta$ is an integer we can easily obtain

\begin{cor}
	Let $\beta\geq2$ be an integer and $\alpha\in [0,\beta-1]$. If $\phi$ satisfies the conditions in Theorem~\ref{theorem main theorem}, then we have
		\begin{align} 
		\dim_HER_\phi^\beta(\alpha)=\begin{cases} h^\beta(\alpha),\indent &0\leq\alpha\leq(\beta-1)/2;\\
			1,\indent &(\beta-1)/2<\alpha\leq \beta-1.\end{cases} 
	\end{align} 
    Here, $h^\beta(\alpha)$ can be reduced to the definition \eqref{definiton h beta alpha}. 
\end{cor}

This paper is organized as follows. The next section is devoted to some notations and basic properties of $\beta$-expansion. In Section 3, we will introduce the lower and upper Besicovitch sets and determine their Hausdorff dimensions, which is prepared for the calculation of the upper bound of Hausdorff dimension of $ER_\phi^\beta(\alpha)$. In Section 4, the full Moran sets and the $\alpha$-Moran sets are introduced, which will be used to obtain the lower bound of Hausdorff dimension of $ER_\phi^\beta(\alpha)$. In Section 5, two essential lemmas are presented for ready use. The last section is devoted to the proof of Theorem~\ref{theorem main theorem} and we will give at first the definition and some basic properties of slowly varying sequence. 

The readers are assumed to be familiar with the definition and basic properties of Hausdorff dimension. The book \cite{F97} of Falconer is highly recommended. For the related topic about the dimensional theory associated with digit average of numbers, one can trace back the history and see the classic work in \cite{Be,Bi,FF,FFW} and the references therein. 

\section{$\beta$-expansion}
In this section, we introduce some notations, definitions and basic properties about $\beta$-expansion, together with some properties about $\beta$-adic entropy function.

First, we would like to give a result of Parry on charactering whether a digit sequence is admissible. For this purpose, we need to introduce the infinite $\beta$-expansion of 1.
Let $\beta>1$ be given. If the $\beta$-expansion of 1, according to~\eqref{formula beta 1}, terminates, in other words there exists an $n\geq1$ such that $\vep_n(1,\beta)\neq0$ but $\vep_m(1,\beta)=0$ for all $m\geq n+1$, then we call $\beta$ a simple Parry number and put
\[(\vep_1^*(1,\beta),\vep_2^*(1,\beta),\ldots)=(\vep_1(1,\beta),\vep_2(1,\beta),\ldots,\vep_n(1,\beta)-1)^\infty.\] 
Here, $(w)^\infty$ denotes the periodic sequence $(w,w,w,\ldots)$ when $w$ is a word. Otherwise, we write $\vep_i^*(1,\beta)=\vep_i(1,\beta)$, $i\geq1$, and use $(\vep_1^*(1,\beta),\vep_2^*(1,\beta),\ldots)$ to denote the $\beta$-expansion of 1. In both cases, the infinite $\beta$-expansion of 1 is set as
\[\vep^*(1,\beta)=(\vep_1^*(1,\beta),\vep_2^*(1,\beta),\ldots).\]
Let $\prec$ be the lexicographical order on $\Sigma^\infty$ which is defined as:
\[(\vep_1,\vep_2,\vep_3,\ldots)\prec(\eta_1,\eta_2,\eta_3,\ldots)\]
if and only if $\vep_1<\eta_1$ or there exists $n\geq1$ such that $\vep_i=\eta_i$ for $1\leq i<n$ but $\vep_n<\eta_n$. 

The admissible sequences is characterized in by the following theorem.

\begin{thm}[Parry \cite{Pa}]\label{theorem parry}
  \noindent	
  \begin{enumerate}
    \item A non-negative integer sequence $(\vep_1,\vep_2,\ldots)$ is $\beta$-admissible if and only if 
    \[(\vep_i,\vep_{i+1},\ldots)\prec(\vep_1^*(1,\beta),\vep_2^*(1,\beta),\ldots),\quad\forall i\geq1.\]
    \item If $1<\beta_1<\beta_2$, then $\Sigma_{\beta_1}\subset\Sigma_{\beta_2}$.
  	\item A non-negative integer sequence $(\vep_1,\vep_2,\ldots)$ is the     expansion of 1 for some $\beta>1$ if and only if 
  	\[(\vep_i,\vep_{i+1},\ldots)\prec(\vep_1,\vep_2,\ldots)\ \text{or}\ (\vep_i,\vep_{i+1},\ldots)=(\vep_1,\vep_2,\ldots),\quad\forall i\geq1.\]
  \end{enumerate}	
\end{thm}

Moreover, the cardinality of $\Sigma_\beta^n$ and the topological entropy of $\beta$-expansion can be charactered by the following theorem.

\begin{thm}[R\'{e}nyi \cite{R}]\label{theorem renyi}
	For any $\beta>1$ and $n\geq1$, we have
	\begin{align}\label{inequality renyi}
		\beta^n\leq\card\Sigma_\beta^n\leq\frac{\beta^{n+1}}{\beta-1}.
	\end{align}
	In particular, the topological entropy of the dynamical system $(I, T_\beta)$ is equal to $\lim_{n\to\infty}(\log\card\Sigma_\beta^n)/n=\log\beta$.
\end{thm}

In what follows, we will introduce a crucial subset $B_0$ of $(1,\infty)$ (see~\cite{LW}). In fact, many problems about $\beta$-expansion can be more easily dealt with when $\beta$ takes value in this set and by using the technique of approximation we can then solve the problems for the general $\beta>1$.

Let $l_n(\beta)$, $n\geq1$, be the length of the longest consecutive zeros following the digit $\vep_n^*(1,\beta)$. That is,
\[l_n(\beta)=\max\{k\geq0\colon\vep_{n+j}^*(1,\beta)=0,\ \text{for all}\ 1\leq j\leq k\},\quad n\geq1.\]
Write
\begin{align}\label{definition b0}
	B_0=\big\{\beta>1\colon \{l_n(\beta)\}_{n\geq1}\ \text{is bounded}\big\}.
\end{align}
Then the set $B_0$ is just the collection $C_3$ in \cite{S} such that $S_\beta$ satisfies the specification property. Moreover, we have

\begin{lem}[See \cite{LTWW,S}]\label{lemma A zero}
	The set $B_0$ is uncountable and dense in $(1,\infty)$. In addition, we have that $\mathcal{L}(B_0)=0$ and $\dim_HB_0=1$. 
\end{lem}

Let $(\vep_1,\ldots,\vep_n)$, $n\geq1$, be an admissible word. We call 
\[I_n(\vep_1,\ldots,\vep_n)=\{x\in I\colon \vep_i(x,\beta)=\vep_i, 1\leq i\leq n\}\]
an $n$-th \emph{cylinder}. Since each cylinder is an interval, we can also call it an $n$-th order \emph{basic interval}. Denote by $I_n(x)=I_n(\vep_1(x,\beta),\ldots,\vep_n(x,\beta))$ the $n$-th cylinder containing $x$ and $|I_n(x)|$ the length of $I_n(x)$. It is easy to see from the expansion \eqref{formula beta 1} that 
\[|I_n(x)|\leq \beta^{-n},\quad n\geq1.\] 
Obviously, we have $I_n(x)|=\beta^{-n}$ when $\beta$ is an integer. In addition to this, the full cylinder of rank $n$ in base $\beta$, defined in the following, also meets this situation.  
\begin{dfn}
	Let $(\vep_1,\ldots,\vep_n)\in\Sigma_\beta^n$, $n\geq1$. An $n$-th cylinder $I_n(\vep_1,\ldots,\vep_n)$ is said to be \emph{full} if its length verifies
	\[|I_n(\vep_1,\ldots,\vep_n)|=\beta^{-n}.\]
	Accordingly, $(\vep_1,\ldots,\vep_n)$ is called a \emph{full word}.
\end{dfn}

Let $\vep=(\vep_1,\ldots,\vep_n)\in\Sigma_\beta^n$, $n\geq1$; $\eta=(\eta_1,\ldots,\eta_m)\in\Sigma_\beta^m$, $m\geq1$. We defined the concatenation of $\vep$ and $\eta$ as
\[\vep\ast\eta=(\vep_1,\ldots,\vep_n,\eta_1,\ldots,\eta_m)\]
if the concatenated word is admissible. The following lemmas characterize the full cylinders.

\begin{lem}[See Lemma 3.1 in \cite{FW}]\label{lemma 31}
Let $\beta>1$ and $\vep=(\vep_1,\ldots,\vep_n)\in\Sigma_\beta^n$, $n\geq1$. The following are equivalent:
\begin{enumerate}
	\item $I_n(\vep_1,\ldots,\vep_n)$ is a full cylinder;
	\item $T_\beta^nI_n(\vep_1,\ldots,\vep_n)=I$;
	\item For any $\eta=\in\Sigma_\beta^m$ with $m\geq 1$, the concatenated word $\vep\ast\eta$ is admissible.
\end{enumerate}	
\end{lem}

\begin{lem}[See Lemma 3.2 and Corollary 3.3 in \cite{FW}]\label{lemma 32}
	Let $\beta>1$ and $m,n\geq1$.
	\begin{enumerate}
		\item If $I_n(\vep_1,\ldots,\vep_n)$ is full, then for any word $(\eta_1,\ldots,\eta_m)\in\Sigma_\beta^m$, we have
		\[|I_{n+m}(\vep_1,\ldots,\vep_n,\eta_1,\ldots,\eta_m)|=|I_n(\vep_1,\ldots,\vep_n)|\cdot|I_m(\eta_1,\ldots,\eta_m)|;\]
		\item Let $p\in\mathbb{N}$. Then
		\[I_{n+p}(\vep_1,\ldots,\vep_n,0^p)\ \text{is full}\Leftrightarrow I_n(\vep_1,\ldots,\vep_n)\geq\beta^{-(n+p)}.\]	
		Here, $0^p$ with $p\geq1$ is a word of length $p$ composed by $0$'s and $0^0$ is the empty word.
	\end{enumerate}	
\end{lem}

By Lemma~\ref{lemma 32} (1) and the definition of full word, we have
\begin{cor}\label{corollary full word}
	Let $\beta>1$ and $m,n\geq1$. If the two cylinders $I_n(\vep_1,\ldots,\vep_n)$ and $I_m(\eta_1,\ldots,\eta_m)$ are full, then the cylinder $I_{n+m}(\vep_1,\ldots,\vep_n,\eta_1,\ldots,\eta_m)$ is full; equivalently, if the two words $(\vep_1,\ldots,\vep_n)$ and $(\eta_1,\ldots,\eta_m)$ are full, then the concatenation $(\vep_1,\ldots,\vep_n,\eta_1,\ldots,\eta_m)$ is full.
\end{cor}

\begin{lem}[See \cite{LW}]\label{lemma LW} 
	Let $\beta>1$ and $n\geq1$. Write $M_n(\beta)=\max_{1\leq i\leq n}\{l_i(\beta)\}$. Then, for any admissible word $(\vep_1,\ldots,\vep_n)$, the cylinder $I_{n+m}(\vep_1,\ldots,\vep_n,0^{m+1})$ is full if $m\geq M_n(\beta)$.
\end{lem}

Since the sequence $\{l_n(\beta)\}$ is bounded for each parameter $\beta$ in $B_0$, by Lemma~\ref{lemma 31}, Lemma~\ref{lemma 32} and Lemma~\ref{lemma LW} we can easily obtain that

\begin{lem}\label{lemma 0 m}
	Let $\beta\in B_0$ and $(\vep_1,\ldots,\vep_n)$, $n\geq1$, be any admissible word. There exists an integer $M>0$ such that $I_{n+M}(\vep_1,\ldots,\vep_n,0^M)$ is full, which leads to that
	\begin{enumerate}
		\item any admissible word $(\eta_1,\ldots,\eta_m)$ can be concatenated behind $(\vep_1,\ldots,\vep_n,0^M)$;
		\item the length of cylinder $I_n(\vep_1,\ldots,\vep_n)$ satisfies  \[\beta^{-(n+M)}\leq|I_n(\vep_1,\ldots,\vep_n)|\leq\beta^{-n}.\]
	\end{enumerate}	
\end{lem}

Recall the root $\beta_m=\beta_m(\beta)$ defined in the equation~\eqref{definition beta m}. Define
\[B_1(\beta)=\left\{\beta_m(\beta)\colon\text{the root of \eqref{definition beta m} where}\ \vep_m^\ast(1,\beta)\geq1\right\}\]
and
\[B_1=\bigcup_{\beta>1}\big(B_1(\beta)\setminus\{1\}\big).\]
Then we have
\begin{lem}\label{lemma b1b0}
	$B_1\subset B_0$ and $B_1$ is dense in $(1,\infty)$.
\end{lem}
\begin{proof}
	It is followed by the fact that the digit sequence of expansion of 1 under base $\beta_m$ is the $m$-periodic sequence $(\vep_1^\ast(1,\beta),\ldots,\vep_{m-1}^\ast(1,\beta),\vep_m^\ast(1,\beta)-1)^\infty$, where the integer $m$ satisfies $\vep_m^\ast(1,\beta)\geq1$. The second conclusion is obvious since $\beta_m\to\beta$ as $m\to\infty$.
\end{proof}

We point out that the formula \eqref{definiton h alpha} may turn into a simpler form when $\beta\in B_1$ as below:
\begin{align}\label{definiton h beta alpha}
	h^\beta(\alpha)=\lim_{\delta\to0}\liminf_{n\to\infty}\frac{\log h^\beta(\alpha,n,\delta)}{(\log \beta)n}=\lim_{\delta\to0}\limsup_{n\to\infty}\frac{\log h^\beta(\alpha,n,\delta)}{(\log \beta)n}.
\end{align}
Moreover, we would like to collect more in the following lemma.

\begin{lem}[See Proposition 4.2 and Proposition 4.4 in \cite{TWWX}]\label{lemma tan}
	Let $\beta>1$.
	\begin{enumerate}
		\item If $\beta_m=\beta_m(\beta)\in B_1$, $m\geq1$, then
		\begin{align*}
			h^\beta(\alpha)&=\lim_{\delta\to0}\lim_{m\to\infty}\liminf_{n\to\infty}\frac{\log h^{\beta_m}(\alpha,n,\delta)}{(\log \beta)n}=\lim_{\delta\to0}\lim_{m\to\infty}\limsup_{n\to\infty}\frac{\log h^{\beta_m}(\alpha,n,\delta)}{(\log \beta)n}\\
			&=\lim_{\delta\to0}\liminf_{n\to\infty}\frac{\log h^\beta(\alpha,n,\delta)}{(\log \beta)n}=\lim_{\delta\to0}\limsup_{n\to\infty}\frac{\log h^\beta(\alpha,n,\delta)}{(\log \beta)n}.
		\end{align*} 
	 	\item $h^\beta(\alpha)$ is a concave function and continuous on $I_\beta$.
	\end{enumerate}
\end{lem}

The second conclusion in the above lemma also indicates that the function $h^\beta(\alpha)$ is increasing on $[0,\alpha^\ast(\beta)]$ and decreasing on $[\alpha^\ast(\beta),\Lambda(\beta)]$. These properties will be used in the following sections when the classified discussions for the corresponding proofs are needed.

\section{Besicovitch sets}
In this section, we will introduce the notions of the lower and upper algebraic average of $\beta$-expansion of numbers. Based on them,  we will give the definitions of lower and upper Besicovitch sets and then determine their Hausdorff dimensions.  

Let $\beta>1$ and $x\in I$. Recall the notation $A_n(x,\beta)=S_n(x,\beta)/n$, $n\geq1$, in the first section. Then denote respectively by
\begin{align}\label{definition axb}
	\b{A}(x,\beta)=\liminf_{n\to\infty}A_n(x,\beta)\quad\text{and}\quad \bar{A}(x,\beta)=\limsup_{n\to\infty}A_n(x,\beta),
\end{align}
the lower and upper algebraic averages of $\beta$-expansion of $x$. If $\b{A}(x,\beta)=\bar{A}(x,\beta)$, then the common value is just the algebraic average $A(x,\beta)=\lim_{n\to\infty}S_n(x,\beta)/n$ being stated in \eqref{AnnSn}. Let $\alpha\in I_\beta$. Define the following Besicovitch set in $\beta$-expansion
\[E^\beta(\alpha)=\{x\in I\colon A(x,\beta)=\alpha\}.\]
Then, for the size of $E^\beta(\alpha)$ we have
\begin{lem}[See Corollary 1.3 in \cite{LL}]\label{lemma e beta alpha}
	Let $\beta>1$ and $\alpha\in I_\beta$. Then  \[\dim_HE^\beta(\alpha)=h^\beta(\alpha).\]
\end{lem}

Recall the definition of $\alpha^\ast(\beta)$ in \eqref{Ax beta}. Since $A(x,\beta)=\alpha^\ast(\beta)$ for almost all $x\in I$ by Birkhoff's ergodic theory, we have 
\[\dim_HE^\beta(\alpha^\ast(\beta))=\dim_H\{x\in I\colon A(x,\beta)=\alpha^\ast(\beta)\}=1.\]
Thus, Lemma~\ref{lemma e beta alpha} gives that
\begin{cor}\label{corollary ast beta}
	Let $\beta>1$. Then we have $h^\beta(\alpha^\ast(\beta))=1$.
\end{cor}

Furthermore, define respectively the lower and upper Besicovitch sets in $\beta$-expansion as follows: 
\[\b{E}^\beta(\alpha)=\{x\in I\colon\b{A}(x,\beta)\geq\alpha\}\quad\text{and}\quad \bar{E}^\beta(\alpha)=\{x\in I\colon\bar{A}(x,\beta)\leq\alpha\},\]
where $\alpha\in I_\beta$. Then, for the sizes of $\b{E}^\beta(\alpha)$ and $\bar{E}^\beta(\alpha)$, we have

\begin{pro}\label{prop E}
    Let $\beta>1$ and $\alpha\in I_\beta$. Then 
	\begin{align} 
	\dim_H\b{E}^\beta(\alpha)=\begin{cases} 1,\indent &0\leq\alpha\leq\alpha^\ast(\beta);\\
		h^\beta(\alpha),\indent &\alpha^\ast(\beta)<\alpha\leq\Lambda(\beta),\end{cases} 
    \end{align}	
and
	\begin{align}\label{bar E alpha} 
	\dim_H\bar{E}^\beta(\alpha)=\begin{cases} h^\beta(\alpha),\indent &0\leq\alpha\leq\alpha^\ast(\beta);\\
		1,\indent &\alpha^\ast(\beta)<\alpha\leq\Lambda(\beta).\end{cases} 
    \end{align}	
\end{pro}

Note that Proposition \ref{prop E} generalizes the classic work in \cite{Be} of Besicovitch on binary expansion. In fact, if $\beta=2$,  then we can easily obtain that $\alpha^*(\beta)=1/2$, $\Lambda(2)=1$, $I_2=[0,1]$ and by the Stirling's approximation we also have
\[h^2(\alpha)=\frac{H(\alpha)}{\log2},\quad\text{where}\ H(\alpha)=-\alpha\log\alpha-(1-\alpha)\log(1-\alpha),\ \alpha\in I_2.\]

To prove Proposition \ref{prop E}, we would like to present firstly a lemma about the relation among the $\beta$-adic entropy function, the $\beta$-adic lower entropy function and the $\beta$-adic upper entropy function. The corresponding notations of them are given in the following.

Let $\beta>1$, $\alpha\in I_\beta$, $n\geq1$ and $\delta>0$. Denote  
\[\b{H}^\beta(\alpha,n,\delta)=\left\{(\vep_1,\ldots,\vep_n)\in \Sigma_\beta^n\colon\sum_{i=1}^n\vep_i>n(\alpha-\delta)\right\}\]
and $\b{h}^\beta(\alpha,n,\delta)=\card\b{H}^\beta(\alpha,n,\delta)$. Then define the $\beta$-adic lower entropy function
\begin{align}\label{definition b h}
	\b{h}^\beta(\alpha)=\lim_{\delta\to0}\lim_{m\to\infty}\liminf_{n\to\infty}\frac{\log\b{h}^{\beta_m}(\alpha,n,\delta)}{(\log \beta)n}, 
\end{align}
Here and in the sequel, the number $\beta_m$ with $m\geq1$ is given in \eqref{definition beta m}. Symmetrically, denote  
\[\bar{H}^\beta(\alpha,n,\delta)=\left\{(\vep_1,\ldots,\vep_n)\in \Sigma_\beta^n\colon\sum_{i=1}^n\vep_i<n(\alpha+\delta)\right\}\]
and $\bar{h}^\beta(\alpha,n,\delta)=\card\bar{H}^\beta(\alpha,n,\delta)$. Then define the $\beta$-adic upper entropy function 
\begin{align}\label{definition bar h}
	\bar{h}^\beta(\alpha)=\lim_{\delta\to0}\lim_{m\to\infty}\limsup_{n\to\infty}\frac{\log\bar{h}^{\beta_m}(\alpha,n,\delta)}{(\log \beta)n}.
\end{align}

\begin{lem}\label{lemma relation}
	Let $\beta>1$. 
	\begin{enumerate}
		\item For any $\alpha\in I_\beta$, we have
		\[\b{h}^\beta(\alpha)=\lim_{\delta\to0}\liminf_{n\to\infty}\frac{\log\b{h}^{\beta}(\alpha,n,\delta)}{(\log \beta)n}\quad\text{and}\quad\bar{h}^\beta(\alpha)=\lim_{\delta\to0}\limsup_{n\to\infty}\frac{\log\bar{h}^{\beta}(\alpha,n,\delta)}{(\log \beta)n};\]
		\item If $0\leq\alpha\leq\alpha^*(\beta)$, then $\bar{h}^\beta(\alpha)=h^\beta(\alpha)$; if $\alpha^*(\beta)<\alpha\leq\Lambda(\beta)$, then $\bar{h}^\beta(\alpha)=1$;
		\item If $0\leq\alpha<\alpha^*(\beta)$, then $\b{h}^\beta(\alpha)=1$; if $\alpha^*(\beta)\leq\alpha\leq\Lambda(\beta)$, then $\b{h}^\beta(\alpha)=h^\beta(\alpha)$.
	\end{enumerate}
\end{lem}
\begin{proof}
	(1) It can be derived similar to the proof of Proposition 4.2 in \cite{TWWX}. 
	
	(2) For the first part, it is followed by (2) in Lemma \ref{lemma tan} and the following inequality \[h^\beta(\alpha,n,\delta)\leq\bar{h}^\beta(\alpha,n,\delta)\leq2\left\lceil\frac{\alpha+\delta}{2\delta}\right\rceil h^\beta(\alpha,n,\delta)\] 
	for sufficiently large $n$ and small $\delta$. The second part is followed by the relation $h^\beta(\alpha^*(\beta))\leq\bar{h}^\beta(\alpha)$ and Corollary~\ref{corollary ast beta}. 
	
	(3) It can be dealt with in a similar way as that of (2). 	
\end{proof}	

Now, we are ready to give the proof of Proposition \ref{prop E}.

\begin{proof}[Proof of Proposition \ref{prop E}]
  It suffices to prove the conclusion \eqref{bar E alpha}. Since it is clear that $\dim_H\bar{E}^\beta(\alpha)=1$ when $\alpha^\ast(\beta)<\alpha\leq\Lambda(\beta)$ according to the ergodic theory, we only need to prove that $\dim_H\bar{E}^\beta(\alpha)=h^\beta(\alpha)$ when $0\leq\alpha\leq\alpha^\ast(\beta)$ in the following.
  
  On the one hand, we have $E^\beta(\alpha)\subset\bar{E}^\beta(\alpha)$ for any $\alpha\in I_\beta$. It yields that
  \begin{align}\label{formula e 1}
    \dim_H\bar{E}^\beta(\alpha)\geq\dim_HE^\beta(\alpha)=h^\beta(\alpha)
  \end{align} 
  by Lemma~\ref{lemma e beta alpha}.	
	
  On the other hand, for any $\delta>0$, we have
  \[\bar{E}^\beta(\alpha)\subset\bigcap_{l=1}^\infty\bigcup_{n=l}^\infty\bigcup_{(\vep_1,\ldots,\vep_n)\in \bar{H}^\beta(\alpha,n,\delta)}I_n(\vep_1,\ldots,\vep_n),\]
  By (1) in Lemma~\ref{lemma relation}, for any $\eta>0$, there exists an integer $N$ such that 
  \[\bar{h}^\beta(\alpha,n,\delta)<\beta^{n\left(\bar{h}^\beta(\alpha)+\frac{\eta}{2}\right)},\quad\forall n>N.\]	
  Then, for any $l>N$, the $(\bar{h}^\beta(\alpha)+\eta)$-Hausdorff measure of $\bar{E}^\beta(\alpha)$ satisfies
  \[\mathbb{H}_{\beta^{-l}}^{\bar{h}^\beta(\alpha)+\eta}\big(\bar{E}^\beta(\alpha)\big)\leq\sum_{n=l}^{\infty}\bar{h}^\beta(\alpha,n,\delta)(\beta^{-n})^s<\sum_{n=l}^{\infty}(\beta^{-\frac{\eta}{2}})^n<\infty,\]	
  which implies that $\dim_H\bar{E}^\beta(\alpha)\leq\bar{h}^\beta(\alpha)+\eta$. Thus,
  \begin{align}\label{formula e 2}
  	\dim_H\bar{E}(\alpha)\leq\bar{h}^\beta(\alpha)=h^\beta(\alpha)
  \end{align}  
  by the arbitrariness of $\eta$ and (2) in Lemma~\ref{lemma relation}.	
	
  On combining \eqref{formula e 1} and \eqref{formula e 2}, it finishes the proof.	
\end{proof}

\section{Moran sets}
In this section, we first recall the structure and a dimensional result about homogeneous Moran sets, then introduce the definition and some properties about full Moran sets and $\alpha$-Moran sets.

\subsection{Homogeneous Moran sets}
Let $\delta>0$ and $\{N_k\}_{k\geq1}$ be a sequence of integers and $\{c_k\}_{k\geq1}$ be a sequence of positive numbers satisfying $N_k\geq2$, $0<c_k<1$, $k\geq1$; and  $N_1c_1\leq\delta$, $N_kc_k\leq1$, $k\geq2$. Put  
\[D_0=\{\emptyset\},\ D_k=\{(i_1,\ldots,i_k)\colon 1\leq i_j\leq N_j,1\leq j\leq k\}\ \text{for}\ k\geq1,\ D=\bigcup_{k\geq0}D_k.\]
Suppose that $J$ is an interval of length $\delta$. A collection $\mathcal{F}=\{J_\sigma\colon \sigma\in D\}$ of subintervals of $J$ is said to have \emph{homogeneous structure} if it satisfies 
\begin{enumerate}
	\item $J_{\emptyset}=J$;
	\item For any $\sigma\in D_{k-1}$ with $k\geq1$, $J_{\sigma\ast j}$, $1\leq j\leq N_k$, are subintervals of $J_\sigma$ and $\text{int}(J_{\sigma\ast i})\cap\text{int}(J_{\sigma\ast j})=\emptyset$ if $i\neq j$, where $\text{int}$ denotes the interior of some set;
	\item For any $\sigma\in D_{k-1}$ with $k\geq1$, we have
	\[\frac{|J_{\sigma\ast j}|}{J_\sigma}=c_j,\quad 1\leq j\leq N_k.\]
\end{enumerate} 
If the collection $\mathcal{F}$ is of homogeneous structure, then the set
\begin{align}\label{definition mj}
	\mathcal{M}=\mathcal{M}\big(J,\{N_k\}_{k\geq1},\{c_k\}_{k\geq1}\big)=\bigcap_{k\geq1}\bigcup_{\sigma\in D_k}J_\sigma
\end{align} 
is called a \emph{homogeneous Moran set} determined by $\mathcal{F}$.

Moreover, write
\[s=\liminf_{k\to\infty}\frac{\log (N_1N_2\cdots N_k)}{-\log(c_1c_2\cdots c_{k+1}N_{k+1})}.\]
Then we have
\begin{lem}[See Theorem 2.1 and Corollary 2.1 in \cite{FWW}]\label{lemma FWW}
	Let $\mathcal{M}$ be the homogeneous Moran set defined in~\eqref{definition mj}. Then we have $\dim_H\mathcal{M}\geq s$. 
    In addition, if $\inf_{k\geq1}c_k>0$, then $\dim_H\mathcal{M}=s$.
\end{lem}

\subsection{Full Moran sets}
Let $\beta\in B_0$, $N\geq1$ be sufficiently large and $P$ be an integer satisfying $0\leq P\leq \Lambda(\beta) (N+M)$. Recall the integer $M$ being given in Lemma \ref{lemma 0 m}. Write
\[W^\beta(P,N+M):=\left\{(\vep_1,\vep_2,\ldots,\vep_N,0^M)\in \Sigma_\beta^{N+M}\colon \sum_{i=1}^N\vep_i=P\right\}.\]
Based on the set $W^\beta(P,N+M)$, define  
\begin{align*}
	&\mathcal {W}^\beta(P,N+M)\\
	&=\big\{x\in I\colon \left(\vep_{i(N+M)+1}(x,\beta),\ldots,\vep_{(i+1)(N+M)}(x,\beta)\right)\in W^\beta(P,N+M), i\geq0\big\}\\
	&=:W^\beta(P,N+M)^\infty.
\end{align*}
That is, $\mathcal {W}^\beta(P,N+M)$ is the set of numbers of which the digit sequences consist of the words in $W^\beta(P,N+M)$. 
Moreover, this definition is valid according to (3) in Lemma~\ref{lemma 31}, Corollary~\ref{corollary full word} and Lemma~\ref{lemma 0 m} (1). Also, in this paper we call $\mathcal {W}^\beta(P,N+M)$ a \emph{full Moran set} since each word $(\vep_1,\vep_2,\ldots,\vep_N,0^M)$ in $W^\beta(P,N+M)$ is full. The following lemma describes the size of $\mathcal {W}^\beta(P,N+M)$.  
\begin{lem}\label{lemma pnm} 
	Let $\beta\in B_0$ and $0\leq P\leq \Lambda(\beta)(N+M)$, $N\geq1$. Then
	\[\dim_H\mathcal {W}^\beta(P,N+M)=\frac{\log\card W^\beta(P,N+M)}{(\log \beta)(N+M)}.\]
\end{lem}
\begin{proof}
	Since $(\vep_1,\vep_2,\ldots,\vep_N,0^M)$ is full, we have $|(\vep_1,\vep_2,\ldots,\vep_N,0^M)|=\beta^{-(N+M)}$. Moreover, for any $k\geq1$, by Corollary~\ref{corollary full word}, we have
	\[|(\vep_1,\vep_2,\ldots,\vep_N,0^M)^k|=\beta^{-k(N+M)}.\]
	This, together with Lemma \ref{lemma FWW}, leads to the conclusion
	\begin{align*}
		\dim_H\mathcal {W}^\beta(P,N+M)&=\liminf_{k\to\infty}\frac{k\log\card W^\beta(P,N+M)}{(\log \beta)(k+1)(N+M)-\log\card W^\beta(P,N+M)}\\
		&=\frac{\log\card W^\beta(P,N+M)}{(\log \beta)(N+M)},
	\end{align*}
which ends the proof.	
\end{proof}	

\begin{cor}\label{corollary wpnm}
	Let $\beta\in B_0$, $\alpha\in I_\beta$ and $N\geq1$. Then we have
	\begin{align}
		\dim_H\mathcal {W}^\beta( [\alpha(N+M)],N+M)=\frac{\log\card W^\beta( [\alpha(N+M)],N+M)}{(\log \beta)(N+M)}.
	\end{align}
\end{cor}
\begin{proof}
	Take $P= [\alpha(N+M)]$ in Lemma \ref{lemma pnm}.
\end{proof}

Moreover, we can even obtain

\begin{lem}\label{lemma logcard}
	Let $\beta\in B_0$ and $\alpha\in I_\beta$. We have
	\[\lim_{N\to\infty}\dim_H\mathcal {W}^\beta( [\alpha(N+M)],N+M)=h^\beta(\alpha).\]
\end{lem}
\begin{proof}
	According to Corollary \ref{corollary wpnm}, we first show that
	\begin{align}\label{equality lemma logcard}
		\liminf_{N\to\infty}\frac{\log\card W^\beta([\alpha(N+M)],N+M)}{(\log \beta)(N+M)}=h^\beta(\alpha).
	\end{align}
	The proof is divided into three cases: $0\leq\alpha<\alpha^*(\beta)$, $\alpha=\alpha^*(\beta)$ and $\alpha^*(\beta)<\alpha\leq \Lambda(\beta)$. Here, we give only the proof of the first case. The other two cases can be dealt with in a similar way.
	
	Define a function with two variables $\alpha$ and $N$:
	\begin{align}
		w_M^\beta(\alpha,N)=\card W^\beta([\alpha(N+M)],N+M).
	\end{align}
	Then, by Lemma~\ref{lemma tan} we know that for each $\alpha$, $w_M^\beta(\alpha,N)$ is increasing with respect to $N$; for each $N$, $w_M^\beta(\alpha,N)$ is constant on $[(k-1)/(N+M),k/(N+M))$ where $1\leq k\leq \Lambda(\beta)(N+M)$; and, in addition, for each $N$, $w_M^\beta(\alpha,N)$ is increasing on $[0,\alpha^\ast(\beta)+1/(N+M))$ and decreasing on $[\alpha^\ast(\beta)+1/(N+M),\Lambda(\beta)]$ with respect to $\alpha$.
	
	Take $\delta>0$ such that $\alpha+\delta<\alpha^*(\beta)$. Since 
	\[w_M^\beta(\alpha+\delta,N)\leq\frac{(N+M)^{(N+M)\delta}}{((N+M)\delta-1)!}w_M^\beta(\alpha,N),\]
	by the Stirling's approximation we have 
	\begin{align*}
		\begin{split}
			\liminf_{N\to\infty}\frac{\log w_M^\beta(\alpha+\delta,N)}{(\log \beta)(N+M)} &\leq\liminf_{N\to\infty}\frac{\log\frac{(N+M)^{(N+M)\delta}}{(N+M)(\delta-1)!}}{(\log\beta)(N+M)}+\liminf_{N\to\infty}\frac{\log w_M^\beta(\alpha,N)}{(\log\beta)(N+M)}\\
			&=\frac{\delta\log\frac{e}{\delta}}{\log \beta}+\liminf_{N\to\infty}\frac{\log w_M^\beta(\alpha,N)}{(\log \beta)(N+M)}. 
		\end{split}
	\end{align*}	
	Let $\delta\to0$ in both sides and by the fact $\lim_{\delta\to0}(\delta\log\frac{e}{\delta})/(\log\beta)=0$, it yields that
	\begin{align}\label{inequality log halpha}
		\lim_{\delta\to0}\liminf_{N\to\infty}\frac{\log w_M^\beta(\alpha+\delta,N)}{(\log\beta)(N+M)}\leq\liminf_{N\to\infty}\frac{\log w_M^\beta(\alpha,N)}{(\log\beta)(N+M)}.
	\end{align}
	Moreover, we have
	\begin{align*}
			w_M^\beta(\alpha+\delta,N)\leq h^\beta(\alpha,N+M,\delta)
			\leq([2\delta(N+M)]+1) w_M^\beta(\alpha+\delta,N).  
	\end{align*}
	It follows that
	\[\lim_{\delta\to0}\liminf_{N\to\infty}\frac{\log w_M^\beta(\alpha+\delta,N)}{(\log\beta)(N+M)}=\lim_{\delta\to0}\liminf_{N\to\infty}\frac{\log h^\beta(\alpha,N+M,\delta)}{(\log\beta)(N+M)}=h^\beta(\alpha).\]	
	This, together with \eqref{inequality log halpha}, implies
	\[h^\beta(\alpha)\leq\liminf_{N\to\infty}\frac{\log w_M^\beta(\alpha,N)}{(\log \beta)(N+M)}.\]
	On the other hand, the inequality for the other direction is apparently true since
	\[\liminf_{N\to\infty}\frac{\log w_M^\beta(\alpha,N)}{(\log \beta)(N+M)}\leq\liminf_{N\to\infty}\frac{\log w_M^\beta(\alpha+\delta,N)}{(\log\beta)(N+M)}\]
	by the properties of $w_M^\beta(\alpha,N)$. So, the equality \eqref{equality lemma logcard} is established.
	
	Since the equality 
	\[\limsup_{N\to\infty}\frac{\log w_M^\beta(\alpha,N)}{(\log \beta)(N+M)}=h^\beta(\alpha)\]
	can be proved just as the above discussion, the proof is ended here.
\end{proof}	

Take $P=[\alpha(N+M)]+1$ in Lemma~\ref{lemma pnm}, we may also obtain
\begin{cor}\label{formula nm1}
	Let $\beta\in B_0$ and $\alpha\in I_\beta$. Then   \[\lim_{N\to\infty}\dim_H\mathcal {W}^\beta([\alpha(N+M)]+1,N+M)=h^\beta(\alpha).\]
\end{cor}
\begin{proof}
	By lemma~\ref{lemma logcard}, we have
	\begin{align*}
		&\lim_{N\to\infty}\dim_H\mathcal {W}^\beta([\alpha(N+M)]+1,N+M)\\
		&=\lim_{N\to\infty}\dim_H\mathcal {W}^\beta\Big(\Big[\big(\alpha+\frac{1}{N+M}\big)(N+M)\Big],N+M\Big)\\
		&=\lim_{N\to\infty}h^\beta\Big(\alpha+\frac{1}{N+M}\Big)=h^\beta(\alpha)
	\end{align*}
    The last equality is followed by the continuity of $h^\beta(\alpha)$ in (2) of Lemma~\ref{lemma tan}.
\end{proof} 

More generally, let $\beta\in B_0$, $N\geq1$ be large enough and take two integers $P$ and $Q$ satisfying $0\leq P\leq Q\leq \Lambda(\beta) (N+M)$. Write
\[W^\beta([P,Q],N+M):=\left\{(\vep_1,\vep_2,\ldots,\vep_N,0^M)\in \Sigma_\beta^{N+M}\colon P\leq\sum_{i=1}^N\vep_i\leq Q\right\}.\]
Based on the set $W^\beta([P,Q],N+M)$, define the full Moran set 
\begin{align*}
	&\mathcal {W}^\beta([P,Q],N+M)\\
	&=\big\{x\in I\colon \left(\vep_{i(N+M)+1}(x,\beta),\ldots,\vep_{(i+1)(N+M)}(x,\beta)\right)\in W^\beta([P,Q],N+M), i\geq0\big\}\\
	&=:W^\beta([P,Q],N+M)^\infty.
\end{align*}
Then, by an analogue discussion to that of Lemma~\ref{lemma pnm} we may obtain
\begin{lem}\label{lemma pqnm} 
	Let $\beta\in B_0$ and the two integers $P$ and $Q$ satisfy $0\leq P\leq Q\leq\Lambda(\beta) (N+M)$, $N\geq1$. Then
	\begin{align}\label{formula pqnm1}
		\dim_H\mathcal {W}^\beta([P,Q],N+M)=\frac{\log\card W^\beta([P,Q],N+M)}{(\log \beta)(N+M)}.
	\end{align} 
\end{lem}

As a matter of fact, for the above formula \ref{formula pqnm1}, $P$ and $Q$ are not necessarily integers. In a more detail, define similarly
\begin{align}
	\mathcal {W}^\beta((P,Q),N+M):=W^\beta((P,Q),N+M)^\infty
\end{align}
where
\[W^\beta((P,Q),N+M):=\left\{(\vep_1,\vep_2,\ldots,\vep_N,0^M)\in \Sigma_\beta^{N+M}\colon P<\sum_{i=1}^N\vep_i<Q\right\},\]
then we also have
\begin{lem}\label{formula pqnm2}
	Let $\beta\in B_0$ and two numbers $P$ and $Q$ satisfy $0\leq P< Q\leq\Lambda(\beta) (N+M)$, $N\geq1$. Then
	\[\dim_H\mathcal {W}^\beta((P,Q),N+M)=\frac{\log\card W^\beta((P,Q),N+M)}{(\log \beta)(N+M)}.\]
\end{lem}

\subsection{$\alpha$-Moran sets}
In this subsection, we will construct a set $\mathcal{W}_\infty^\beta(\alpha,N+M)$, called $\alpha$-Moran set, where $0\leq\alpha<\Lambda(\beta)$, $\beta\in B_0$ and $N>1$. It can be used to construct a suitable Moran subset to obtain the lower bound of Hausdorff dimension of $ER_\phi^\beta(\alpha)$. To this end, we first construct recursively two sequences of sets of digit words $\{W_n^\beta(\alpha,N+M)\}_{n=1}^\infty$ and $\{V_n^\beta(\alpha,N+M)\}_{n=1}^\infty$. Take $N$ to be sufficiently large such that 
\[[\alpha (N+M)]+1<\Lambda(\beta)(N+M).\]
Put
\[W_1^\beta(\alpha,N+M)=\left\{(\vep_1,\ldots,\vep_N,0^M)\in\Sigma_\beta^{N+M}\colon\sum_{i=1}^{N}\vep_i=[\alpha(N+M)]\right\}\]
and
\[V_1^\beta(\alpha,N+M)=\left\{(\vep_1,\ldots,\vep_N,0^M)\in\Sigma_\beta^{N+M}\colon\sum_{i=1}^{N}\vep_i=[\alpha(N+M)]+1\right\}.\]
Suppose that the sets $W_i^\beta(\alpha,N+M)$ and $V_i^\beta(\alpha,N+M)$ are well defined for all $1\leq i\leq n$, then define 
\begin{align*}
	\begin{split}
		W&_{n+1}^\beta(\alpha,N+M)=\big\{(\vep_1,\vep_2,\ldots,\vep_{2^n(N+M)})\in W\big([\alpha 2^n(N+M)],2^n(N+M)\big)\colon\\
		&(\vep_{2^{n-1}(N+M)i+1},\ldots,\vep_{2^{n-1}(N+M)(i+1)})\in W_n^\beta(\alpha,N+M)\cup V_n^\beta(\alpha,N+M),i=0,1\big\}
	\end{split}
\end{align*}
and
\begin{align*}
	\begin{split}
		V&_{n+1}^\beta(\alpha,N+M)=\big\{(\vep_1,\vep_2,\ldots,\vep_{2^n(N+M)})\in W\big([\alpha 2^n(N+M)]+1,2^n(N+M)\big)\colon\\
		&(\vep_{2^{n-1}(N+M)i+1},\ldots,\vep_{2^{n-1}(N+M)(i+1)})\in W_n^\beta(\alpha,N+M)\cup V_n^\beta(\alpha,N+M),i=0,1\big\}.
	\end{split}
\end{align*}
The above definitions are valid since the estimation
\begin{equation}\label{cc5}
2[\alpha2^n(N+M)]<[\alpha2^{n+1}(N+M)]+1\leq 2\big([\alpha2^n(N+M)]+1\big)
\end{equation}
holds for all $n\geq0$.

With this construction, we know that for each $n\geq1$, every word in $W_n^\beta(\alpha,N+M)$ is of length $2^{n-1}(N+M)$ and the sum of digits is $[\alpha2^{n-1}(N+M)]$. Likewise, every word in $V_n^\beta(\alpha,N+M)$ is of length $2^{n-1}(N+M)$ and the sum of digits is $[\alpha2^{n-1}(N+M)]+1$. 
Moreover, we may even obtain

\begin{lem}\label{remark decompose}
	Let $\beta\in B_0$ and $0\leq\alpha<\Lambda(\beta)$. For any $0\leq i\leq n-1$, we can decompose uniquely each word in $W_n^\beta(\alpha,N+M)$ or $V_n^\beta(\alpha,N+M)$ into successive concatenations of digit words of length $2^i(N+M)$, the sum of digits in each one is either $[\alpha 2^i(N+M)]$ or $[\alpha 2^i(N+M)]+1$.
\end{lem} 

Based on the sequence of $\{W_n^\beta(\alpha,N+M)\}_{n=1}^\infty$, define the so-called $\alpha$-Moran set in $\beta$-expansion
\begin{align*}
	&\mathcal{W}_\infty^\beta(\alpha,N+M)\\
	&=\big\{x\in I\colon\left(\vep_{(2^{n-1}-1)(N+M)+1}(x,\beta),\ldots,\vep_{(2^n-1)(N+M)}(x,\beta)\right)\in W_n^\beta(\alpha,N+M),n\geq1\big\}\\
	&=:\prod_{n=1}^\infty W_n^\beta(\alpha,N+M).
\end{align*}
That is, $\mathcal{W}_\infty^\beta(\alpha,N+M)$ is the set of numbers of which the digit sequences concatenate the words in $W_n^\beta(\alpha,N+M)$, $n\geq1$, one by one in the order of natural numbers. 

\begin{lem}\label{lemma alpha M}
	Let $\beta\in B_0$ and $0\leq\alpha<\Lambda(\beta)$. Then
	\begin{align}\label{equation alpha M}
		\lim_{N\to\infty}\dim_H\mathcal{W}_\infty^\beta(\alpha,N+M)=h^\beta(\alpha).
	\end{align}
\end{lem}
\begin{proof}
	For the case $0\leq\alpha<\alpha^*(\beta)$, take $N$ to be large enough such that 
	\[[\alpha(N+M)]+1<[\alpha^*(\beta)(N+M)].\] 
	Then by the structures of sequences in $\mathcal{W}_\infty^\beta(\alpha,N+M)$, described in Lemma~\ref{remark decompose}, we have
	\begin{equation}
	\begin{split}\label{inequality alpha N}
		&\dim_H\mathcal{W}^\beta([\alpha (N+M)],N+M)\leq\dim_H\mathcal{W}_\infty^\beta(\alpha,N+M)\\
		&\leq\dim_H\mathcal{W}^\beta\big(\big[[\alpha(N+M)],[\alpha(N+M)]+1\big],N+M\big).
	\end{split}
	\end{equation}
   Since
	\begin{align*}
		\begin{split}
			&\dim_H\mathcal{W}^\beta\big(\big[[\alpha(N+M)],[\alpha (N+M)]+1\big],N+M\big)\\ 
			&=\frac{\log\Big( w_M^\beta(\alpha,N)+  w_M^\beta\big(\alpha+\frac{1}{N+M},N\big)\Big)}{(\log\beta)(N+M)}\\
			&\leq\frac{\log\Big(2\max\Big\{ w_M^\beta(\alpha,N),w_M^\beta\big(\alpha+\frac{1}{N+M},N\big)\Big\}\Big)}{(\log\beta)(N+M)},
		\end{split}
	\end{align*}
	by letting $N\to\infty$ we have 
	\begin{align}\label{lemma 471}
		\lim_{N\to\infty}\dim_H\mathcal{W}^\beta\big(\big[[\alpha(N+M)],[\alpha (N+M)]+1\big],N+M\big)\leq h^\beta(\alpha) 
	\end{align}
    according to Lemma~\ref{lemma logcard} and Corollary \ref{formula nm1}. Moreover, we have
    \begin{align*}\label{lemma 472}
    	 \lim_{N\to\infty}\dim_H\mathcal{W}^\beta([\alpha (N+M)],N+M)=h^\beta(\alpha) 
    \end{align*}
    by Lemma~\ref{lemma logcard}. This, together with \eqref{inequality alpha N} and \eqref{lemma 471}, leads to the conclusion~\eqref{equation alpha M}.
	
	On the other hand, for the remainder case $\alpha^*(\beta)\leq\alpha<\Lambda(\beta)$, we have similarly that
	\begin{equation*}
	\begin{split} 
	&\min\left\{\dim_H\mathcal{W}^\beta([\alpha (N+M)],N+M),\dim_H\mathcal{W}^\beta([\alpha (N+M)]+1,N+M)\right\}\\
	&\leq\dim_H\mathcal{W}_\infty^\beta(\alpha,N+M)\\
	&\leq\dim_H\mathcal{W}^\beta\big(\big[[\alpha(N+M)],[\alpha(N+M)]+1\big],N+M\big).
	\end{split}
	\end{equation*}
	Then we can prove this case by Corollary \ref{formula nm1} and the same discussion as that for the foregoing case.
	
	The proof is completed now.
\end{proof}

\section{Two lemmas}

In this section, we will present two lemmas for the proof of Theorem~\ref{theorem main theorem}, i.e., the following Lemma \ref{lemma density} and Lemma \ref{lemma main theorem}. The first one is related to subsets of $\mathbb{N}$ with zero density and the second one describes the lower bound of Hausdorff dimension of $ER_\phi^\beta(\alpha)$.

Let $\mathbb{M}$ be a subset of $\mathbb{N}$ and write its complementary set of  $\mathbb{N}$ as
$\mathbb{N}\backslash\mathbb{M} =\{n_1<n_2<\ldots\}$. Suppose there are a set $D\subset I$ and a mapping
$\phi_\mathbb{M} \colon D\to I$ such that the corresponding digit sequences satisfy
\[(\vep_1(x,\beta),\vep_2(x,\beta),\ldots)\in\Sigma_D\mapsto(\vep_{n_1}(x,\beta),\vep_{n_2}(x,\beta),\ldots)\in\Sigma_\beta,\] 
where $\Sigma_D$ denotes the set of digit sequences of numbers in $D$, i.e., \[\Sigma_D=\{(\vep_1(x,\beta),\vep_2(x,\beta),\ldots)\in\Sigma_\beta\colon x\in D\}.\]
If there exist such pair of set $D$ and mapping $\phi_\mathbb{M}$, then we call the mapping $\phi_\mathbb{M}$ \emph{maps well} on the set $D$.
Given such set $D\subset I$ and mapping $\phi_\mathbb{M}$, we may obtain another set
\[\phi_\mathbb{M} (D)=\{\phi_\mathbb{M} (x)\colon x\in D\}.\] 
In addition, we call the set $\mathbb{M} $ is of \emph{density zero} in $\mathbb{N}$ if
\[\lim_{n\to\infty}\frac{\card\{i\in\mathbb M\colon i\leq n\}}{n}=0.\]
Then the relation between the sizes of $D$ and $\phi_\mathbb{M} (D)$ can be described as follows.

\begin{lem}\label{lemma density}
	Let $\beta\in B_0$, $D\subset I$ and $\mathbb{M} $ be of density zero in $\mathbb{N}$. If the mapping $\phi_\mathbb{M}$ maps well on $D$, then we have
	\[\dim_H\phi_\mathbb{M} (D)=\dim_HD.\]
\end{lem}
\begin{proof}
	(1) To show $\dim_H\phi_\mathbb{M} (D)\geq\dim_H D$. If $\dim_H\phi_\mathbb{M} (D)=r$, then for any $s>t>r$, the
	$t$-Hausdorff measure of $\phi_\mathbb{M} (D)$ is zero, i.e.,
	$\mathbb {H}^{t}(\phi_\mathbb{M} (D))=0$. Therefore, there exists a
	$\delta$-cover $\{I_{l_j}(\bar{x}^j)\}_{j\geq1}$ of $\phi_\mathbb{M} (D)$ with $0<\delta<1$ such that
	\begin{equation}
	\sum_{j\geq1}|I_{l_j}(\bar{x}^j)|^{t}<\infty,
	\end{equation}
	where $\bar{x}^j\in\phi_\mathbb{M} (D)$ and
	$I_{l_j}(\bar{x}^j)$ denotes the cylinder $I_{l_j}(\bar{\vep}_1^j,\ldots,\bar{\vep}_{l_j}^j)$. 
	Here,  
	\[(\bar{\vep}_1^j,\ldots,\bar{\vep}_{l_j}^j)=\big(\vep_1(\bar{x}^j,\beta),\ldots,\vep_{l_j}(\bar{x}^j,\beta\big),\quad j\geq1.\]
	It yields that
	\begin{align}\label{mmm}
		\sum_{j\geq1}\beta^{-(l_j+M)t}<\infty 
	\end{align}
	by (2) of Lemma~\ref{lemma 0 m}. For any $x\in D$, assume that
	$\phi_\mathbb{M} (x)=\bar{x}\in\phi_\mathbb{M} (D)$ with digit sequence $(\bar{\vep}_1,\bar{\vep}_2,\ldots)$.
	Write
	$\phi_\mathbb{M} (I_N(\vep_1,\ldots,\vep_N))=I_n(\bar{\vep}_1,\ldots,\bar{\vep}_n)$.
	Since
	\[\frac{N-n}{N}=\frac{\card\{i\in\mathbb{M} \colon i\leq
		N\}}{N}\to0\]
	as $N\to\infty$, for the above $\delta$ there exists an integer $N_1$
	such that
	\begin{equation}\label{mm}
	\frac{N-n}{N}<\delta,\quad\mbox{i.e.},\quad n\leq N<\frac{n}{1-\delta}\ \text{for}\ N>N_1.
	\end{equation}
	Moreover, we may even require that
	$\delta$ is small enough such that for all $j\geq1$,
	\begin{equation}\label{mmmm}
	l_j>N_1\quad \mbox{and} \quad \frac{l_j}{l_j+M}\big(s-\frac{\delta}{1-\delta}\big)>t.
	\end{equation}
	Since $\phi_\mathbb{M} (D)\subset\bigcup_{j\geq1}I_{l_j}(\bar{x}^j)$, we
	have that
	$D\subset\bigcup_{j\geq1}\phi_\mathbb{M} ^{-1}\left(I_{l_j}(\bar{x}^j)\right)$. Note that, by \eqref{mm} and \eqref{inequality renyi} in Theorem~\ref{theorem renyi},
	the number of cylinders with rank greater than $l_j$ in $\phi_\mathbb{M} ^{-1}\left(I_{l_j}(\bar{x}^j)\right)$ is less than or equal to $\beta^{\left[\frac{\delta l_j}{1-\delta}\right]+1}/(\beta-1)$. Thus, by \eqref{mmm} and~\eqref{mmmm}, we
	have 
	\begin{align*}
		\mathbb
		{H}^{s}_\delta(D)&\leq\sum_{j\geq1}\frac{\beta^{\left[\frac{\delta l_j}{1-\delta}\right]+1}}{\beta-1}
		\beta^{-l_js}\leq\frac{\beta}{\beta-1}\sum_{j\geq1}\beta^{\frac{\delta l_j}{1-\delta}}
		\beta^{-l_js}\\
		&=\frac{\beta}{\beta-1}\sum_{j\geq1}\beta^{-l_j\left(s-\frac{\delta}{1-\delta}\right)}
		\leq\frac{\beta}{\beta-1}\sum_{j\geq1}\big(\beta^{-(l_j+M)}\big)^{\frac{l_j}{l_j+M}\big(s-\frac{\delta}{1-\delta}\big)}\\
		&\leq\frac{\beta}{\beta-1}\sum_{j\geq1}\beta^{-(l_j+M)t}<\infty.
	\end{align*}
	It follows that $\mathbb
	{H}^{s}(D)<\infty$ and then $\dim_HD<s$. Since $s>r$ is arbitrary, we
	obtain that
	$\dim_HD\leq r=\dim_H\phi_\mathbb{M} (D)$.
	
	(2) To show $\dim_H\phi_\mathbb{M} (D)\leq\dim_H D$. For any
	arbitrary $0<\epsilon<1$, since the set
	$\mathbb{M} \subset\mathbb{N}$ is of density zero, we can choose an integer $N_0>M$
	such that 
	\[\frac{\card\{i\in\mathbb{M} \colon i\leq
		n\}}{n}<\epsilon\quad\mbox{for all}\ n\geq N_0.\]
	Take two numbers $x$ and $y$ with $d(x,y)=\beta^{-t}$, where $N\leq t<N+1$ for some $N\geq N_0$.
	Then, by (2) of Lemma~\ref{lemma 0 m}, we have 
	\begin{align*}
		d(\phi_\mathbb{M} (x),\phi_\mathbb{M} (y))&\leq \beta^{-N+\card\{i\in\mathbb{M} \colon i\leq
			N\}}=\left(\beta^{-N}\right)^{1-\frac{\card\{i\in\mathbb{M} \colon i\leq N\}}{N}}\\
		&<\left(\beta^{-(N+1)}\right)^{\frac{N}{N+1}(1-\epsilon)}<d(\vep,\eta)^{\frac{N}{N+1}(1-\epsilon)}.
	\end{align*}
	It means that the mapping $\phi_\mathbb{M}$ is $N(1-\epsilon)/(N+1)$-H\"{o}lder on $D$. So, by Proposition 2.3 in~\cite{F97}, we have \[\dim_H\phi_\mathbb{M} (D)<\frac{N+1}{N(1-\epsilon)}\dim_HD.\]
	This leads to the conclusion $\dim_H\phi_\mathbb{M} (D)\leq\dim_HD$ by the arbitrariness of $\epsilon$ and $N$.
	
	Finally, by combining the above two assertions, we finish the proof.
\end{proof}

Lemma \ref{lemma density} tells us that the Hausdorff dimension of a set will be invariant if the set of deleted positions of digit sequences of numbers is of density zero in $\mathbb{N}$. Certainly, we need to ensure that each new digit sequence, by deleting the set of positions with density zero from the original sequence, is admissible.

Before the presentation of the second lemma, we need to introduce the definition and corresponding properties of slowly varying sequences.

\begin{dfn}[See~\cite{GS,Seneta,BGT}]\label{definition slow increase}
	Let $\theta$ be a function satisfying $\theta(n)>0$ for all $n\geq1$. We call the sequence $\{\theta(n)\}_{n\geq1}$ \emph{slowly varying} if there is a sequence of positive numbers $\{f(n)\}_{n\geq1}$ satisfying
	\begin{align}
		\lim_{n\to\infty}\frac{\theta(n)}{f(n)}=K>0
	\end{align}
	and
	\begin{align}\label{formula slow increase}
		\lim_{n\to\infty}n\left(1-\frac{f(n-1)}{f(n)}\right)=0.
	\end{align}
\end{dfn}

There are some typical slowly varying sequences such as  \[\{C>0\}_{n\geq1},\{\log n\}_{n\geq1},\{\log\log n\}_{n\geq1},\{\arctan n\}_{n\geq1},\{\exp(\ln^\nu n)\}_{n\geq1},\]
where $0<\nu<1$, etc. The following is a list of some basic properties of the slowly varying sequences.

\begin{lem}[See Lemma 2.2 in \cite{CDL}]\label{lem increse}
	Let the sequence $\{\theta(n)\}_{n\geq1}$ be slowly varying. Then
	\begin{enumerate}
		\item\label{inrease 1} the sequence $\{C\theta(n)\}_{n\geq1}$, where $C>0$, is also slowly varying;
		\item\label{inrease 2} $\lim_{n\to\infty}\log \theta(n)/\log n=0$;
		\item\label{inrease 3} $\lim_{n\to\infty} \theta(n)/n=0$.
	\end{enumerate}
\end{lem}

Now, we are ready to present the second lemma which just explores the lower bound of Hausdorff dimension of $ER_\phi^\beta(\alpha)$ for the special case $\beta\in B_0$. Based on it, in the last section we will prove Theorem~\ref{theorem main theorem} for general $\beta>1$ using the method of approximation. 

\begin{lem}\label{lemma main theorem}
	Let $\alpha\in I_\beta$, where $\beta\in B_0$. Let the sequence $\{\theta(n)\}_{n\geq1}$ be slowly varying and $\theta(n)\to\infty$ as $n\to\infty$. If $\phi(n)=[\theta(n)]$, $n\geq1$, then we have
	\[\dim_HER_\phi^\beta(\alpha)\geq h^\beta(\alpha)\quad\text{as}\ 0\leq\alpha\leq\alpha^\ast(\beta)\]
	and 
	\[\dim_HER_\phi^\beta(\alpha)=1\quad\text{as}\  \alpha^\ast(\beta)<\alpha\leq\Lambda(\beta).\]
\end{lem}
\begin{proof}
We will first prove that $\dim_HER_\phi^\beta(\alpha)\geq h^\beta(\alpha)$ when $0\leq\alpha\leq\alpha^\ast(\beta)$ by constructing a suitable subset of $ER_\phi^\beta(\alpha)$ which is related to $\alpha$-Moran sets. And then prove respectively that $\dim_HER_\phi^\beta(\alpha)=1$ when $\alpha^\ast(\beta)<\alpha<\Lambda(\beta)$ and when $\alpha=\Lambda(\beta)$ by constructing Moran subsets with sufficiently large Hausdorff dimensions. 	
	
Case I: $0\leq\alpha\leq\alpha^\ast(\beta)$. Similar to the definition of  $\mathcal{W}_\infty^\beta(\alpha,N+M)$, define the set of numbers
\begin{align}
  \mathcal{W}_+^\beta(\alpha,N+M):=W_1^\beta(\alpha,N+M)\times\prod_{n=1}^{\infty}W_n^\beta(\alpha,N+M).	
\end{align}	
Then we may obtain  
\begin{align}\label{relation subset}
	\mathcal{W}_+^\beta(\alpha,N+M)\subset ER_\phi^\beta(\alpha).
\end{align} 
In fact, for any number $x\in\mathcal{W}_+^\beta(\alpha,N+M)$ and integer $r\geq1$, we can decompose its digit sequence $(\vep_1(x,\beta),\vep_2(x,\beta),\ldots)$ into successive concatenations of digit words of length $2^r(N+M)$. Moreover, the sum of digits in each word is $[\alpha2^r(N+M)]$ or $[\alpha2^r(N+M)]+1$ but the initial one. Fix $r\geq1$, assume that
\[K2^r(N+M)\leq\phi(n)<(K+1)2^r(N+M)\]
for some integer $K$, then
\begin{align}\label{inequality knm1} 
	\frac{(K-2)[\alpha 2^r(N+M)]}{(K+1)2^r(N+M)}\leq A_{n,\phi(n)}(x,\beta)\leq\frac{(K+1)([\alpha 2^r(N+M)]+1)}{K2^r(N+M)}.
\end{align} 
Let $n\to\infty$, then $\phi(n)\to\infty$ and $K\to\infty$. Thus,
\begin{align}\label{inequality knm2}  
	\frac{[\alpha 2^r(N+M)]}{2^r(N+M)}\leq\b{A}_\phi(x,\beta)\leq
	\bar{A}_\phi(x,\beta)\leq\frac{[\alpha 2^r(N+M)]+1}{2^r(N+M)},
\end{align} 
where
\[\b{A}_\phi(x,\beta)=\liminf_{n\to\infty}A_{n,\phi(n)}(x,\beta)\quad\text{and}\quad\bar{A}_\phi(x,\beta)=\limsup_{n\to\infty}A_{n,\phi(n)}(x,\beta).\]
Let $r\to\infty$, then we have $A_\phi(x,\beta)=\alpha$, 
which proves \eqref{relation subset}. 

Consequently, we have
\begin{align} 
	\dim_H\mathcal{W}_+^\beta(\alpha,N+M)\leq\dim_HER_\phi^\beta(\alpha).
\end{align}
Moreover, it is clear that
\begin{align} 
	\dim_H\mathcal{W}_+^\beta(\alpha,N+M)=\dim_H\mathcal{W}_\infty^\beta(\alpha,N+M)
\end{align}
according to the countable stationarity of Hausdorff dimension. Thus,
\[\dim_HER_\phi^\beta(\alpha)\geq\dim_H\mathcal{W}_\infty^\beta(\alpha,N+M).\]
By letting $N\to\infty$ and Lemma \ref{lemma alpha M}, it yields that
\[\dim_HER_\phi^\beta(\alpha)\geq h^\beta(\alpha).\]	

Case II: $\alpha^\ast(\beta)<\alpha<\Lambda(\beta)$. Since $h^\beta(\alpha^\ast(\beta))=1$ by Corollary~\ref{corollary ast beta}, for any $\epsilon>0$, there exists a number $\delta_0$ satisfying with $0<\delta_0<\alpha^\ast(\beta)-\alpha$ and a positive integer $N_0$ such that 
\[\frac{\log\card\left\{(\vep_1,\ldots,\vep_n)\in\Sigma_\beta^n\colon\left(\alpha^\ast(\beta)-\frac{\delta_0}{2}\right)n<\sum_{i=1}^n\vep_i<\left(\alpha^\ast(\beta)+\frac{\delta_0}{2}\right)n\right\}}{(\log\beta)n}>1-\frac\epsilon2\]
for all $n\geq N_0$. Take an integer $n_0>N_0$. Denote
\begin{align*}
  &U^\beta(\alpha^\ast(\beta),n_0+M,\delta_0)\\
  &=\Big\{(\vep_1,\ldots,\vep_{n_0},0^M)\in\Sigma_\beta^{n_0+M}\colon(\alpha^\ast(\beta)-\delta_0)(n_0+M)<\sum_{i=1}^{n_0}\vep_i<(\alpha^\ast(\beta)+\delta_0)(n_0+M)\Big\}.
\end{align*}
Based on $U^\beta(\alpha^\ast(\beta),n_0+M,\delta_0)$, construct the set of numbers
\[\mathcal{U}^\beta(\alpha^\ast(\beta),n_0+M,\delta_0):=U^\beta(\alpha^\ast(\beta),n_0+M,\delta_0)^\infty.\]
Then, by the formula \eqref{formula pqnm2} and the above estimation, we have
\begin{align}\label{formula un0}
	\dim_H\mathcal{U}^\beta(\alpha^\ast(\beta),n_0+M,\delta_0)=\frac{\log\card U^\beta(\alpha^\ast(\beta),n_0+M,\delta_0)}{(\log\beta)(n_0+M)}>1-\epsilon.
\end{align}
for sufficiently large $n_0$.
 
Next, based on $\mathcal{U}^\beta(\alpha^\ast(\beta),n_0+M,\delta_0)$, we will construct a set $\mathcal{U}^\beta(\alpha^\ast(\beta),n_0+M,\delta_0;\alpha,N+M)$, which is also denoted by $\mathcal{U}^\beta(\alpha,N+M)$ for simplicity, in the following.

For each $x(n_0,\delta_0)\in\mathcal{U}^\beta(\alpha^\ast(\beta),n_0+M,\delta_0)$, we first construct, by induction, a number $x^*(\alpha,N+M)$. Write $x^{(0)}=x(n_0,\delta_0)=(\vep_1(x^0,\beta),\vep_2(x^0,\beta),\ldots)$. Suppose we have defined
\[x^{(j)}=\big(\vep_1(x^{(j)},\beta),\vep_2(x^{(j)},\beta),\ldots),\quad\text{for}\ 0\leq j\leq k,\]
then set
\[x^{(k+1)}=\big(\vep_1(x^{(k)},\beta),\dots,\vep_{l_k}(x^{(k)},\beta), 0^Mw_{k+1}^\beta(\alpha,N+M),\vep_{l_k+1}(x^{(k)},\beta),\ldots\big).\]
Here, $w_{k+1}^\beta(\alpha,N+M)\in W_{k+1}^\beta(\alpha,N+M)$ and
\[l_1=N,\ l_k=\big(4^{k-1}N+(2^{k-1}-2)(N+M)+(k-1)\big)M,\quad k\geq2.\]
That is, $x^{(k+1)}$ is obtained by inserting a word $0^Mw_{k+1}^\beta(\alpha,N+M)$ of length $M+2^k(N+M)$ at the position $l_k+1$ of the digit sequence of $x^{(k)}$. It is easy to see that $\{x^{(j)}\}_{j\geq0}$ is a Cauchy sequence. Denote by
\[x^*(\alpha,N+M)=(\vep_1(x^*,\beta),\vep_2(x^*,\beta),\ldots)\]
the corresponding limit point of $\{x^{(j)}\}_{j\geq0}$. In other words, by inserting the sequence of digit words $\{0^Mw_{k+1}^\beta(\alpha,N+M)\}_{k\geq0}$ into the original positions $4^kN$, $k\geq0$, of the digit sequence of number $x(n_0,\delta_0)$, we obtain  $x^*(\alpha,N+M)$.

Then define the set
\begin{align}
	\mathcal{U}^\beta(\alpha,N+M)=\big\{x^*(\alpha,N+M)\in I\colon x(n_0,\delta_0)\in\mathcal{U}^\beta(\alpha^\ast(\beta),n_0+M,\delta_0)\big\}.
\end{align}
It has the following properties:

(1) The set of positions occupied by the sequence $\big\{0^Mw_{k+1}^\beta(\alpha,N+M)\big\}_{k\geq0}$ in each digit sequence of number of $\mathcal{U}^\beta(\alpha^\ast(\beta),n_0,\delta_0)$ is of density zero in $\mathbb{N}$. This is due to the estimation  
\begin{align*}
	&\limsup_{n\to\infty}\frac{\card\big\{i\leq n\colon \vep_i^*(\alpha,N+M)\ \text{is a digit in some}\ 0^Mw_k^\beta(\alpha,N+M), k\geq1\big\}}{n}\\
	&\leq\limsup_{t\to\infty}\frac{(N+M)+\cdots+2^t(N+M)+tM}{l_t}\\
	&=0.
\end{align*}

Now, define the mapping
\begin{align*}
	\phi(\alpha,n_0,\delta_0)\colon\mathcal{U}^\beta(\alpha,N+M)&\to\mathcal{U}^\beta(\alpha^\ast(\beta),n_0+M,\delta_0)\\
	x^*(\alpha,N+M)&\mapsto x(n_0,\delta_0), 
\end{align*}
then the above property (1) implies that
\begin{align}\label{property a}
	\dim_H\mathcal{U}^\beta(\alpha,N+M)=\dim_H\mathcal{U}^\beta(\alpha^\ast(\beta),n_0+M,\delta_0)
\end{align}
by Lemma~\ref{lemma density}.

(2) $\mathcal{U}^\beta(\alpha,N+M)\subset ER_\phi^\beta(\alpha)$. Take $x^*\in\mathcal{U}^\beta(\alpha,N+M)$. Note that every word $w_{k+1}^\beta(\alpha,N+M)$ in $W_{k+1}^\beta(\alpha,N+M)$ can be decomposed into successively concatenated words of length $N+M$, the sum of digits in each word is $[\alpha(N+M)]$ or $[\alpha(N+M)]+1$. Thus, when the word $(\vep_{i+1},\vep_{i+2},\ldots,\vep_{i+\phi(n)})$ appears in some word $w_{k+1}^\beta(\alpha,N+M)$ in the digit sequence of $x^*(\alpha,N+M)$,
\[I_{n,\phi(n)}(x,\beta)=\max_{0\leq i\leq n-\phi(n)}\{S_{i+\phi(n)}(x,\beta)-S_i(x,\beta)\}\]
can reach its maximal value. In fact, this is guaranteed by an estimation of the value of $\phi(n)$ given in the following.

Since the sequence $\{\theta(n)\}_{n\geq1}$ is slowly varying and $\lim_{n\to\infty}\theta(n)=\infty$, we have
\[\lim_{n\to\infty}\frac{\log\phi(n)}{\log n}=\lim_{n\to\infty}\frac{\log[\theta(n)]}{\log n}=\lim_{n\to\infty}\frac{\log \theta(n)}{\log n}=0\]
by the property (2) in Lemma~\ref{lem increse}. Thus, there exists a number $L>0$ such that
\[\frac{\log\phi(n)}{\log n}\ll\frac{1}{2},\quad \text{i.e.},\quad \phi(n)\ll n^{\frac{1}{2}},\quad \forall n>L.\]
Take $N$ to be sufficiently large such that
\[(N+M)^{\frac{1}{2}}>\max\{2,\sqrt{L}\}.\]
If $l_k<n\leq l_{k+1}$ for some integer $k$, then
\begin{align*}
	\phi(n)&\leq\phi\big(4^k(N+M)+(2^k-2)(N+M)+kM)\big)\\
	&<\phi\left(4^{k+1}(N+M)\right)\ll\left(4^{k+1}(N+M)\right)^{\frac{1}{2}}\\
	&=2^{k+1}(N+M)^{\frac{1}{2}}\\
	&<2^k(N+M).
\end{align*}
It means that the length of digit word $(\vep_{i+1},\ldots,\vep_{i+\phi(n)})$ is far less than that of the word $w_{k+1}^\beta(\alpha,N+M)$ for sufficiently large $N$.

For any $r\geq1$, assume that $K2^r(N+M)\leq\phi(n)<(K+1)2^r(N+M)$ for some integer $K$.  Then we have
\[\frac{(K-1)[\alpha 2^r(N+M)]}{(K+1)2^r(N+M)}\leq A_{n,\phi(n)}(x,\beta)\leq\frac{(K+1)([\alpha 2^r(N+M)]+1)}{K2^r(N+M)}.\]
Let $n\to\infty$, then $\phi(n)\to\infty$ and $K\to\infty$. And then let $r\to\infty$, similar to the discussion in the inequalities \eqref{inequality knm1} and \eqref{inequality knm2}, we may obtain $A_\phi(x,\beta)=\alpha$. 
It leads to the conclusion $\mathcal{U}^\beta(\alpha,N+M)\subset ER_\phi^\beta(\alpha)$.

Property (2) implies that
\[\dim_HER_\phi^\beta(\alpha)\geq\dim_H\mathcal{U}^\beta(\alpha,N+M).\]
This, together with \eqref{formula un0} and \eqref{property a}, yields that
$\dim_HER_\phi^\beta(\alpha)>1-\epsilon$. It proves this case since $\epsilon$ is arbitrary.

Case III: $\alpha=\Lambda(\beta)$. The technique for the proof of this case is similar to that of Case II and we would like to only give the outline here. First, by the definition of $\Lambda(\beta)$, for any $j\geq1$, there exists digit sequence $\eta^{(j)}=(\eta_1^{(j)},\eta_2^{(j)},\ldots)\in\Sigma_\beta$ and an integer $n_j$ such that
\[\Lambda(\beta)-\frac{1}{j}<\frac{\sum_{i=1}^{n_j}\eta_i^{(j)}}{n_j}<\Lambda(\beta)+\frac{1}{j}.\]
Moreover, the sequence $\{n_j\}_{j\geq1}$ can be chosen to be strictly increasing. Then, according to the sequences $\{\eta^{(j)}\}_{j\geq1}$ and $\{n_j\}_{j\geq1}$, construct the following sequence of admissible words:
\[(0^M,\eta_1^{(1)},0^M),(0^M,\eta_1^{(1)},\eta_2^{(1)},0^M),\ldots,(0^M,\eta_1^{(1)},\eta_2^{(1)},\ldots,\eta_{n_1}^{(1)},0^M)\]
\[(0^M,\eta_1^{(2)},\eta_2^{(2)},\ldots,\eta_{n_1+1}^{(2)},0^M),\ldots,(0^M,\eta_1^{(2)},,\eta_2^{(2)},\ldots,\eta_{n_2}^{(2)},0^M),\]
\[\ldots,\ \ldots,\ \ldots,\]
\[(0^M,\eta_1^{(j+1)},\eta_2^{(j+1)},\ldots,\eta_{n_j+1}^{(j)},0^M),\ldots,(0^M,\eta_1^{(j+1)},,\eta_2^{(j+1)},\ldots,\eta_{n_{j+1}}^{(j+1)},0^M),\ldots.\]

Next, for each $x\in\mathcal{U}^\beta(\alpha^\ast(\beta),n_0+M,\delta_0)$ which is defined in the previous Case II, construct the sequence $x\big(\{\eta^{(j)}\},\{n_j\}\big)$ by inserting the above sequence of words into the positions $2^kN$, $k\geq1$, of the digit sequence $x$. In other words, in this case we have
\[l_k=2^kN+2(k-1)M+\frac{(k-1)k}{2},\quad k\geq1.\]
Next, define the Moran set $\mathcal{U}^\beta\big(\alpha^\ast(\beta),n_0+M,\delta_0;\{\eta^{(j)}\},\{n_j\},N+M\big)$, being denoted by  $\mathcal{U}^\beta\big(\{\eta^{(j)}\},\{n_j\},N+M\big)$ for short, as 
\[\mathcal{U}^\beta\big(\{\eta^{(j)}\},\{n_j\},N+M\big)=\left\{x\big(\{\eta^{(j)}\},\{n_j\}\big)\in\Sigma_\beta\colon x\in\mathcal{U}^\beta(\alpha^\ast(\beta),n_0+M,\delta_0)\right\}.\]
Then we can deduce that
\[\dim_H\mathcal{U}^\beta\big(\{\eta^{(j)}\},\{n_j\},N+M\big)=\dim_H\mathcal{U}^\beta(\alpha^\ast(\beta),n_0+M,\delta_0)\]
by Lemma~\ref{lemma density} and \[\mathcal{U}^\beta\big(\{\eta^{(j)}\},\{n_j\},N+M\big)\subset ER_\phi^\beta(\Lambda(\beta)).\] 
It follows that 
\[\dim_H ER_\phi^\beta(\Lambda(\beta))\geq\dim_H\mathcal{U}^\beta(\alpha^\ast(\beta),n_0+M,\delta_0)>1-\epsilon.\]  
Thus, we have $\dim_H ER_\phi^\beta(\Lambda(\beta))=1$ since $\epsilon$ is arbitrary.

The proof is finished now.	
\end{proof} 

\section{Proof of Theorem \ref{theorem main theorem}}

This section is devoted to the proof of Theorem~\ref{theorem main theorem}. In what follows, we would like to introduce a mapping $\pi_\beta$ at first. 

Let $S_\beta$ be the closure of $\Sigma_\beta$ under the product topology on $\Sigma^\infty$ and $\sigma$ be the shift operator on it which is defined as
\[\sigma(\vep_1,\vep_2,\vep_3,\ldots)=(\vep_2,\vep_3,\vep_4,\ldots)\]
for any $(\vep_1,\vep_2,\vep_3,\ldots)\in\Sigma^\infty$. Then $(S_\beta,\sigma|_{S_\beta})$ is a subshift of the symbolic space $(\Sigma^\infty,\sigma)$ and the two systems $(S_\beta,\sigma|_{S_\beta})$ and $(I,T_\beta)$ are metrically isomorphism. Then define the mapping $\pi_\beta\colon S_\beta\to I$ as
\begin{align}\label{definition pi}
	\pi_\beta(\vep)=\sum_{i=1}^{\infty}\frac{\vep_i}{\beta^i},\quad\text{where}\ \vep=(\vep_1,\vep_2,\ldots)\in S_\beta.
\end{align}
It is easy to see that $\pi_\beta$ is a one-to-one mapping for all but countable many digit sequences. In fact, the two digit sequences
\[(\vep_1,\vep_2,\ldots,\vep_n,0,0,\ldots)\quad\text{and}\quad\big(\vep_1,\vep_2,\ldots,\vep_n-1,\vep_1^\ast(x,\beta),\vep_2^\ast(x,\beta),\ldots\big)\]
share the same image, where the word $(\vep_1,\vep_2,\ldots,\vep_n)$, with $\vep_n\geq1$ and $n\geq1$, is admissible and $\big(\vep_1^\ast(x,\beta),\vep_2^\ast(x,\beta),\ldots\big)$ is the $\beta$-expansion of 1.
Moreover, the mapping $\pi_\beta$ is continuous on $S_\beta$ and satisfies $\pi_\beta\circ\sigma=T_\beta\circ\pi_\beta$. 

Now, we are ready to give the proof of Theorem \ref{theorem main theorem} in which the approximation technique is applied to obtain the lower bound of Hausdorff dimension of $ER_\phi^\beta(\alpha)$. One can see the applications of this technique, for instance, in~\cite{TW} and in the proof of Theorem 1.2 in~\cite{LL}.  

\begin{proof}[Proof of Theorem \ref{theorem main theorem}]
	Let $\beta>1$. For the case $0\leq\alpha\leq\alpha^\ast(\beta)$, we will show, respectively, that the upper bound and lower bound of Hausdorff dimension of $ER_\phi^\beta(\alpha)$ are of common value $h^\beta(\alpha)$.  In addition, for the case $\alpha^\ast(\beta)<\alpha\leq\Lambda(\beta)$, we will show that $\dim_HER_\phi^\beta(\alpha)$ is bigger than $\log\beta_m/\log\beta$ for all $m\geq1$.
	
	\textbf{Upper bound}. Let $x\in ER_\phi^\beta(\alpha)$. Since $A_\phi(x,\beta)=\alpha$, for any $\epsilon>0$ there exists an integer $N_0>0$ such that
	\[A_{n,\phi(n)}(x,\beta)<\alpha+\epsilon,\quad \forall n>N_0.\]
	Fix $n_0>N_0$. Since $\phi(n)\to\infty$ as $n\to\infty$, we can take $m$ to be sufficiently large such that 
	\[m>N_0\quad\text{and}\quad\frac{\phi(n_0)}{\phi(m)}<\epsilon.\] 
	Then, by dividing the beginning digit word $(\vep_1(x,\beta),\vep_2(x,\beta),\ldots,\vep_m(x,\beta))$ of $x$ into $[m/\phi(m)]+1$ successive digit words and all, except the last one, are of equal lengths $\phi(m)$, we have
	\begin{align*}
		\begin{split}
			\frac{S_m(x,\beta)}{m}&<\frac{\big[\frac{m}{\phi(m)}\big](\alpha+\epsilon)\phi(m)+
				\Big(\displaystyle{\frac{m-\big[\frac{m}{\phi(m)}\big]\phi(m)}{\phi(n_0)}}+1\Big)(\alpha+\epsilon)\phi(n_0)}{m}\\
			&=(\alpha+\epsilon)+(\alpha+\epsilon)\frac{\phi(n_0)}{m}
			<(\alpha+\epsilon)+(\alpha+\epsilon)\frac{\phi(n_0)}{\phi(m)}\\
			&<(\alpha+\epsilon)(1+\epsilon) 
		\end{split}
	\end{align*}
	for any $\alpha\in I_\beta$. It yields that $\bar{A}(x,\beta)\leq\alpha$ by the arbitrariness of $\epsilon$. So, we have $ER_\phi^\beta(\alpha)\subset\bar{E}^\beta(\alpha)$. Then Proposition~\ref{prop E} gives that
	\[\dim_HER_\phi^\beta(\alpha)\leq\dim_H\bar{E}^\beta(\alpha)=h^\beta(\alpha)\]	
	when $0\leq\alpha\leq\alpha^\ast(\beta)$.
	
	\textbf{Lower bound}. Recall the definition of root
	$\beta_m$ given in \eqref{definition beta m}, which satisfies $\beta_m\in B_1\subset B_0$ for $m$ large enough and
	\[\beta_m\leq\beta,\quad \Sigma_{\beta_{m_1}}\subset\Sigma_{\beta_{m_2}}\subset\Sigma_\beta\  \text{if}\ m_1<m_2\quad\text{and}\quad\lim_{m\to\infty}\beta_m=\beta.\]	
	Put $D_{\beta,\beta_m}=\pi_\beta(\Sigma_{\beta_m})$ and define a mapping $g\colon D_{\beta,\beta_m}\to I$ satisfying
	\[g(x)=\pi_{\beta_m}(\vep(x,\beta)),\quad x\in D_{\beta,\beta_m}.\]
	Then we have the following three conclusions:
	\begin{enumerate}
		\item $\vep(g(x),\beta_m)=\vep(x,\beta)$;
		\item $g\big(ER_\phi^\beta(\alpha)\cap D_{\beta,\beta_m}\big)=ER_\phi^{\beta_m}(\alpha)$;
		\item the function $g$ is $(\log\beta_m/\log\beta)$-Lipschitz on $D_{\beta,\beta_m}$.
	\end{enumerate}
	The first two conclusions are obvious and the last conclusion is followed by Theorem 3.1 in \cite{BL} and Lemma~\ref{lemma b1b0}. Thus, we have 
	\[\dim_HER_\phi^\beta(\alpha)\geq\dim_H\big(ER_\phi^\beta(\alpha)\cap D_{\beta,\beta_m}\big)\geq\frac{\log\beta_m}{\log\beta}\dim_HER_\phi^{\beta_m}(\alpha).\]
	By Lemma \ref{lemma main theorem}, we obtain that \[\dim_HER_\phi^\beta(\alpha)\geq\frac{\log\beta_m}{\log\beta}\quad\text{as}\ \alpha^\ast(\beta)<\alpha\leq\Lambda(\beta)\] 
	and 
	\[\dim_HER_\phi^\beta(\alpha)\geq\frac{\log\beta_m}{\log\beta}h^{\beta_m}(\alpha)\quad\text{as}\ 0\leq\alpha\leq\alpha^\ast(\beta).\]
	Let $m\to\infty$, then we have that $\dim_HER_\phi^\beta(\alpha)\geq h^\beta(\alpha)$ as $0\leq\alpha\leq\alpha^\ast(\beta)$ and $\dim_HER_\phi^\beta(\alpha)=1$ as $\alpha^\ast(\beta)<\alpha\leq\Lambda(\beta)$.
	
	The proof is completed now.
\end{proof}

Take two integer functions $\phi(n)=[c\log n]$ and $\phi(n)=[c\arctan n]$, where $n\geq1$ and $c>0$, for examples. Let $\beta>1$. Define
\begin{align}
	ER_{\log}^\beta(\alpha)=\left\{x\in I\colon\lim_{n\to\infty}A_{n,[c\log n]}(x,\beta)=\alpha\right\},\quad \alpha\in I_\beta,
\end{align}
and 
\begin{align}
	ER_{\arctan}^\beta(\alpha)=\left\{x\in I\colon\lim_{n\to\infty}A_{n,[c\arctan n]}(x,\beta)=\alpha\right\},\quad \alpha\in I_\beta.
\end{align}
It is evident that both of $\{c\log n\}_{n\geq1}$ and $\{c\arctan n\}_{n\geq1}$ are slowly varying sequences by (1) in Lemma~\ref{lem increse}. Moreover, since $c\log n\to\infty$ and $c\arctan n\to\infty$ as $n\to\infty$, by Theorem~\ref{theorem main theorem} we have 
\begin{cor}
	Let $\beta>1$ and $\alpha\in I_\beta$. Then
	\begin{align*}
		\dim_HER_{\log}^\beta(\alpha)=\dim_HER_{\arctan}^\beta(\alpha)=\begin{cases} h^\beta(\alpha),\indent &0\leq\alpha\leq\alpha^\ast(\beta);\\
			1,\indent &\alpha^\ast(\beta)<\alpha\leq\Lambda(\beta).\end{cases} 
	\end{align*}
\end{cor}

\subsection*{Acknowledgment}
This work was finished when the author visited the Laboratoire d'Analyse et de Math\'{e}matiques Appliqu\'{e}es, Universit\'{e} Paris-Est Cr\'{e}teil Val de Marne, France. Thanks a lot for the great encouragement and assistance provided by the laboratory.

\end{document}